\documentclass{amsart}
\pdfoutput=1 
\usepackage{amsmath,amssymb,verbatim}

\usepackage[colorlinks]{hyperref}

\usepackage{times}

\usepackage[dvips]{graphicx}
\usepackage{amsfonts}
\usepackage{amsthm}
\usepackage{epsfig}
\usepackage{color}

\usepackage{url}

\newdimen\hoogte    \hoogte=12pt    
\newdimen\breedte   \breedte=14pt   
\newdimen\dikte  \dikte=0.5pt    
\def\beginYoung{
 \begingroup
 \def\vr{\vrule height0.8\hoogte width\dikte depth 0.2\hoogte}
 \def\fbox##1{\vbox{\offinterlineskip
    \hrule height\dikte
    \hbox to \breedte{\vr\hfill##1\hfill\vr}
    \hrule height\dikte}}
 \vbox\bgroup \offinterlineskip \tabskip=-\dikte \lineskip=-\dikte
  \halign\bgroup &\fbox{##\unskip}\unskip  \crcr }
\def\End@Young{\egroup\egroup\endgroup}

\newtheorem{theorem}{Theorem}[section]
\newtheorem{proposition}[theorem]{Proposition}
\newtheorem{lemma}[theorem]{Lemma}
\newtheorem{corollary}[theorem]{Corollary}

\newtheorem*{Murtheorem}{Murnaghan Theorem}
\newtheorem*{LRtheorem}{Murnaghan--Littlewood Theorem}
\newtheorem*{Britheorem}{Monotonicity  Theorem}

\theoremstyle{definition}
\newtheorem*{defi}{Definition}
\newtheorem{exa}{Example}

\renewcommand{\S}{\mathfrak{S}}

\newcommand{\stab}{\operatorname{stab}}
\newcommand{\width}{\operatorname{width}}
\newcommand{\RKron}{\operatorname{RKron}}
\newcommand{\maxsupp}[1]{\max \left\{ #1 \;|\; \text{\rm $\gamma$ partition s.t. $\overline{g}_{\alpha,\beta}^{\gamma}>0$} \right\}}
\newcommand{\minsupp}[1]{\min \left\{ #1 \;|\; \text{\rm $\gamma$ partition s.t. $\overline{g}_{\alpha,\beta}^{\gamma}>0$} \right\}}

\newcommand{\scalar}[2]{\langle #1 \,|\, #2 \rangle}
\newcommand{\Supp}{\textrm{Supp}}
\newcommand{\cut}[1]{\overline{#1}}

\title{The stability of the Kronecker product of Schur functions}
\date{\today}
\author{Emmanuel Briand}
\address{
Emmanuel Briand and Mercedes Rosas,
Departamento de \'Algebra,
Facultad de Matem\'aticas,
Aptdo. de Correos 1160,
41080 Sevilla, Spain.}
\email{ebriand@us.es, mrosas@us.es}

\author{Rosa Orellana}
\address{Rosa Orellana, Dartmouth College, Mathematics Department, 6188 Kemeny Hall, Hanover, NH 03755, USA.}
\email{rosa.c.orellana@dartmouth.edu}

\author{Mercedes Rosas}

\thanks{Mercedes Rosas is supported by a Ram\'on y Cajal Fellowship (MICINN, Spain). Mercedes Rosas and Emmanuel Briand are also supported by Projects MTM2007--64509 (MICINN, Spain) and FQM333 (Junta de Andalucia).}

\begin{document}

\begin{abstract}

In the late 1930's Murnaghan discovered the existence of a
stabilization phenomenon for the Kronecker product of  Schur
functions. For $n$ sufficiently large, the values of the Kronecker
coefficients appearing in the product of two Schur functions of degree
$n$ do not depend on the first part of the indexing partitions, but only
on the values of their remaining parts.    We compute the exact value
of $n$ for which all the coefficients of a Kronecker product of Schur
functions stabilize.  We also compute two new bounds for the 
stabilization of a  sequence of coefficients and show that they 
improve existing  bounds of M. Brion and E. Vallejo.  
\end{abstract}


\maketitle

\section*{Introduction}

The understanding of the \emph{Kronecker coefficients of the symmetric group} (the multiplicities appearing when the tensor product of two irreducible representations of the symmetric group is decomposed into irreducibles; equivalently, the structural constants for the Kronecker product of symmetric functions in the Schur basis)
is a longstanding open  problem. Richard Stanley writes ``One of the main problems in the combinatorial  representation theory of the symmetric group is to obtain a combinatorial  interpretation for the Kronecker coefficients'' \cite{Stanley:vol2}. It is also a source of new challenges such as the problem of describing  the set of non--zero Kronecker coefficients \cite{Oeding}, a problem inherited from quantum information theory  \cite{Klyachko,Christandl:Harrow:Mitchison}. Or  proving that the  positivity of a Kronecker coefficient can be decided in polynomial  time, a problem posed by Mulmuley at the heart of his Geometric  Complexity Theory \cite{GCT6}.

The present work  is part of a series of articles that study another family of nonnegative constants, the \emph{reduced Kronecker coefficients} $\overline{g}_{\mu,\nu}^{\lambda}$, as a way to gain understanding about the Kronecker coefficients ${g}_{\mu,\nu}^{\lambda}$, \cite{Briand:Orellana:Rosas:SH,Briand:Orellana:Rosas:Chamber}. 
In \cite{Briand:Orellana:Rosas:Chamber}, we obtained the first explicit
piecewise quasipolynomial  description of a non--trivial family of Kronecker
coefficients, the Kronecker  coefficients indexed by two two--row
shapes.  This new description allowed us to test several conjectures
of Mulmuley. As  a result, we found a  counterexample
\cite{Briand:Orellana:Rosas:SH}  for the strong version of his SH conjecture \cite{GCT6} on the behavior of the Kronecker coefficients  under stretching of its indices.

The starting point of the investigation presented in this paper is a remarkable stability property for the Kronecker products of Schur functions discovered by Murnaghan \cite{Murnaghan:1938, Murnaghan:1955}. This property is best shown on an example, that will be followed by a precise statement. Denote the Kronecker product of $s_{\lambda}$ and $s_{\beta}$ by $s_{\lambda} \ast s_{\beta}$. Then, 
 {\begin{align*}
s_{2,2}\ast s_{2,2}&= 
s_{4}+
s_{1, 1, 1, 1}+
 \phantom{2} s_{2, 2}
 \\
s_{3,2}\ast s_{3,2}&= 
s_{5}+
s_{2, 1, 1, 1}
+\phantom{2}s_{3, 2}+
s_{4, 1}+s_{3, 1, 1}
+\phantom{2}s_{2, 2, 1} \\
s_{4,2}\ast s_{4,2} &= 
s_6
+s_{3, 1, 1, 1}
+2s_{4, 2}
+s_{5, 1}
+s_{4, 1, 1}
+2s_{3, 2, 1}+s_{2, 2, 2} \\
s_{5,2}\ast s_{5,2} &= 
s_{7}+
s_{4, 1, 1, 1}+
2s_{5, 2}+
s_{6, 1}+s_{5, 1, 1}
+2s_{4, 2, 1}
+s_{3, 2, 2}
+s_{4, 3}
+s_{3, 3, 1} \\
s_{6,2}\ast s_{6,2} &= 
s_8
+s_{5, 1, 1, 1}+
2s_{6, 2}
+s_{7, 1}+
s_{6, 1, 1}
+2s_{5, 2, 1}
+s_{4, 2, 2}
+s_{5, 3}
+s_{4, 3, 1} 
+s_{4, 4}\\
s_{7,2}\ast s_{7,2} &= 
s_{9}
+s_{6, 1, 1, 1}
+2s_{7, 2}
+s_{8, 1}
+s_{7, 1, 1}
+2s_{6, 2, 1}
+s_{5, 2, 2} 
+s_{6, 3}
+s_{5, 3, 1}
+s_{5, 4}\\
s_{\bullet,2}\ast s_{\bullet,2} &= 
s_{\bullet}
+s_{\bullet, 1, 1, 1}
+2s_{\bullet, 2}
+s_{\bullet, 1}+s_{\bullet, 1, 1}
+2s_{\bullet, 2, 1}
+s_{\bullet, 2, 2}
+s_{\bullet, 3}
+s_{\bullet, 3, 1}
+s_{\bullet, 4} 
\end{align*}
}
Given a partition $\alpha=(\alpha_1,\ldots,\alpha_k)$ 
and an integer $n$, we set $\alpha[n]=(n-|\alpha|, \alpha_1, \ldots,
\alpha_k)$. Murnaghan's theorem says that for $n$ big enough the expansions of $s_{\alpha[n]} \ast s_{\beta[n]}$ in the Schur
basis all coincide, except for the first part of the indexing partitions which is determined by the degree, $n$.

In particular, given any three partitions $\alpha$, $\beta$ and $\gamma$, the sequence with general term ${g}_{\alpha[n]\beta[n]}^{\gamma[n]}$ is eventually constant. The reduced Kronecker coefficient $\overline{g}_{\alpha,\beta}^{\gamma}$ is defined as the stable value of this sequence. In our example, we see that 
 $\overline{g}_{(2),(2)}^{(2)}=2$ and  $\overline{g}_{(2),(2)}^{(4)}=1$.

In view of the difficulty of studying the Kronecker coefficients,  it
is surprising to obtain theorems that hold in general.   Regardless of
this,   we present new results of general nature. We find an elegant formula that tells the point $n=\stab(\alpha,\beta)$ at which the expansion of the Kronecker product $s_{\alpha[n]} \ast s_{\beta[n]}$ stabilizes:
\[
\stab(\alpha,\beta)=|\alpha|+|\beta|+\alpha_1+\beta_1
\]
We also find new upper bounds for the point at which the sequence ${g}_{\alpha[n]\beta[n]}^{\gamma[n]}$ becomes constant, improving previously known bounds due to Brion \cite{Brion:Foulkes} and Vallejo \cite{Vallejo}. 
Interestingly, our investigations reduce to maximizing or bounding linear forms on the sets $\Supp(\alpha,\beta)$ of partitions $\gamma$ such that $\overline{g}_{\alpha,\beta}^{\gamma} >0$, where $\alpha$ and $\beta$ are fixed partitions. This connects our research to a current problem of major importance: 
 to describe the cones generated by the indices of the nonzero Kronecker coefficients \cite{Klyachko,Oeding}. Moreover, using Weyl's inequalities for eigenvalues of triples of hermitian matrices \cite{Weyl}, we find the maximum of $\gamma_1$ and upper bounds for all parts $\gamma_k$, among all $\gamma$ in $\Supp(\alpha,\beta)$.

  This paper is organized as follows, in Section \ref{sec:main results} we give a detailed
  description of the main results of this work.  In Section \ref{sec:RKC},  we prove the theorem that 
  allows us to recover the Kronecker coefficients from the reduced Kronecker coefficients. 
  We also give an expression of the reduced Kronecker coefficients in terms of Littlewood-Richardson    
  coefficients and Kronecker coefficients.  The main significance of this expression is that it doesn't   
  involve cancellations and it provides us with a tool to prove most of our main results.   In Section \ref{sec:product},
  we provide a proof for the sharp bound for the stability of the Kronecker product.  In the next 
  section, Section \ref{sec:bounds}, we consider the problem of finding bounds on the rows of $\gamma$, whenever $\overline{g}_{\alpha, \beta}^\gamma>0$.  We prove a theorem for a general upper bound for all rows of $\gamma$ using this theorem we give a sharp bound for $\gamma_1$.   In Section \ref{sec:coeff},  we describe a general
  technique for deriving upper bounds for the stabilization of sequences of coefficients.  Using this technique we  get two new bounds.  We show that one of these bounds improves
  the  bounds of Brion and  Vallejo. Finally,  we  compare our results 
  to existing results in the literature.

  


\section{Preliminaries and Main Results}\label{sec:main results}

Let $\lambda$ be a partition (weakly decreasing sequences of positive integers) of  $n$. Denote by  $V_{\lambda}$  the irreducible  representation of the symmetric group $\S_n$  indexed by  $\lambda$. 
 The Kronecker coefficient $g_{\mu,\nu}^{\lambda}$ is  the multiplicity of $V_{\lambda}$ in the decomposition into  irreducible representations of the tensor product $V_{\mu} \otimes V_{\nu}$. 
  The  Frobenius map identifies the irreducible representations $V_\lambda$ of the symmetric group with the Schur function $s_{\lambda}$. In doing so, it allows us to lift the tensor product of representations of the symmetric group to the setting of symmetric functions.  Accordingly, the Kronecker coefficients ${g}_{\mu,\nu}^{\lambda}$ define the Kronecker product on symmetric functions by setting
\[
s_{\mu}*s_{\nu} = \sum_{\lambda} g_{\mu,\nu}^{\lambda} s_{\lambda}.
\]
The reader is referred  to  \cite{Macdonald} Chapter I or \cite{Stanley:vol2} Chapter 7 for the standard facts in the theory of symmetric functions.

Throughout this paper we follow the standard notation for partitions
found in \cite{Macdonald}.  
If $\lambda =(\lambda_1, \lambda_2, \ldots,\lambda_k)$  is a partition, 
its \emph{parts} are its 
 terms $\lambda_i$. The \emph{weight} of $\lambda$ is defined to be the sum of its parts, and it is denoted by $|\lambda|$.
The number $k$ of (nonzero) parts of $\lambda$ is called its \emph{length} , and denoted by $\ell(\lambda).$

We identify a partition $\lambda$ with its  Ferrers
 diagram
 \[
D(\lambda) = \left\{ (i,j) : 1 \le i \le \lambda_j, 1 \leq j \leq \ell(\lambda) \right\} \subseteq \mathbb{N}^2
 \]
This way, we obtain that
$\alpha\, \cap\,  \beta\, $=$( \min(\alpha_1,\beta_1),$
$\min(\alpha_2,\beta_2), \ldots)$.
The sum of two partitions 
$\alpha + \beta$ is defined  as $ ( \alpha_1+\beta_1, \alpha_2+\beta_2,
\ldots)$.

Listing the number of points in each column of $D(\lambda)$ gives
 the  transpose partition of $\lambda$, denoted by
$\lambda'$; equivalently,
one obtains the Ferrers diagram of $\lambda'$  by reflecting the one of
$\lambda$ along its main diagonal.

The skew shape
$\mu / \nu$ is defined as the set difference $D(\mu)\setminus D(\nu)$.
Notice that  
$D(\mu)\subset D(\lambda)$ if $\mu_i\leq
\lambda_i$ for all $i$. 
Again, the  intersection and union of skew-shapes is defined as
the corresponding operations on their diagrams.
The  \emph{width} of $\mu / \nu $ is defined as the number of nonzero columns 
of  $\mu / \nu$  in $\mathbb{N}^2.$

Consider a partition $\lambda$ and an integer $n$. Then $\bar{\lambda}$ is defined to be the partition $(\lambda_2,\lambda_3, \ldots)$ and 
$\lambda[n]$ as the sequence $(n-|\lambda|, \lambda_1, \lambda_2, \ldots)$. 
Notice that $\lambda[n]$ is  a partition only if $n-|\lambda|\geq
\lambda_1$.


We are ready to describe the starting point of our investigations, a  remarkable theorem of  Murnaghan that deserves to be
better known.  We first need to extend the definition of 
$s_{\mu}$ to the case where  $\mu$ is any finite sequence of $n$ integers. For this, we use the  Jacobi-Trudi determinant,
\begin{equation}\label{eq:JT}
s_{\mu}=\det\left(  h_{\mu_j+i-j} \right)_{1\le i,j \le n},
\end{equation}
where $h_k$ is the complete homogeneous symmetric function of degree $k$. In particular, $h_k=0$ if $k$ is negative, and $h_0=1$. It is not hard to see that such a Jacobi--Trudi determinant $s_{\mu}$ is either zero or $\pm 1$ times a Schur function.

\begin{Murtheorem}[Murnaghan, \cite{Murnaghan:1938, Murnaghan:1955}] 
\label{thm:murnaghan}
There exists a family of non-negative integers $(\overline{g}_{\alpha\beta}^\gamma)$ indexed by triples of partitions $(\alpha,\beta,\gamma)$ such that, for  $\alpha$ and $\beta$ fixed, only finitely many terms $\overline{g}_{\alpha\beta}^\gamma$ are nonzero, and for all $n\geq 0$, 
\begin{equation}\label{eq:thm Mur}
s_{\alpha[n]}\ast s_{\beta[n]} =\sum_{\gamma} \overline{g}_{\alpha\beta}^\gamma s_{\gamma[n]}
\end{equation}
Moreover, the coefficient $\overline{g}_{\alpha\beta}^\gamma$ vanishes unless the weights of the three partitions fulfill   the inequalities:  
  \begin{align*} \label{muriq}
  |\alpha|\leq |\beta|+|\gamma|, \qquad |\beta|\leq |\alpha|+|\gamma|,\qquad |\gamma|\leq |\alpha|+|\beta|.
  \end{align*}
  \end{Murtheorem}
In what follows, we refer to these inequalities as \emph{Murnaghan's  inequalities} and we will denote $\Supp(\alpha,\beta)$ the set of all partitions $\gamma$ such that $\overline{g}_{\alpha,\beta}^{\gamma}>0$. We follow 
 Klyachko \cite{Klyachko} and call the coefficients $\overline{g}_{\alpha\beta}^\gamma$ the {\em reduced Kronecker coefficients}.
An elegant  proof of Murnaghan's Theorem, using vertex operators on symmetric functions, is given in \cite{Thibon}.
  
\begin{exa}
According to Murnaghan's theorem the reduced Kronecker coefficients
determine the Kronecker product of two Schur functions, even for small
values of $n$. 
For instance,
\[
s_{2,2}\ast s_{2,2} = 
s_{4}
+s_{1, 1, 1, 1}
+2s_{2, 2}
+s_{3, 1}+s_{2, 1, 1}
+2s_{1, 2, 1}
+s_{0, 2, 2}
+s_{1, 3}
+s_{0, 3, 1}
+s_{0, 4} 
\]
The Jacobi-Trudi determinants corresponding to $s_{1, 2, 1}$ and
$s_{0, 2, 2}$ have a repeated column, hence they are zero. On the other hand, it
is easy to see that $s_{1, 3}=-s_{2,2}$, 
$s_{0, 3, 1}=-s_{2,1,1}$, and $s_{0, 4}=-s_{3,1}$. After taking into
account the resulting cancellations, we recover the
expression of the Kronecker product $s_{2,2}\ast s_{2,2}$ in the Schur basis 
$
s_{4}+s_{1, 1, 1, 1}+s_{2, 2}.
$
\end{exa}

The reduced Kronecker coefficients contain the  Little\-wood--Richardson coefficients as special cases. 
  \begin{LRtheorem}[Murnaghan \cite{Murnaghan:1955}, Littlewood \cite{Littlewood:1958}] Let $\alpha$, $\beta$ and $\gamma$ be partitions.
If $|\gamma|=|\alpha|+|\beta|$, then  the reduced Kronecker coefficient $\overline{g}_{\alpha,\beta}^{\gamma}$ is equal to the Littlewood--Richardson coefficient $c_{\alpha,\beta}^{\gamma}$.
\end{LRtheorem}

Finally, a remarkable result of Christandl, Harrow, and Mitchison
(originally stated for the Kronecker coefficients) says that  the set
\[
\RKron_k= \{(\alpha, \beta, \gamma)\, |\,
\ell(\alpha),\ell(\beta),\ell(\gamma) \le k \text{ and }
\overline{g}_{\alpha,\beta}^\gamma>0\}
\]
is a
finitely generated semigroup under componentwise addition,  \cite{Christandl:Harrow:Mitchison}. That is, if
$\overline{g}_{\alpha, \beta}^{\gamma} \neq 0$ and $\overline{g}_{\hat\alpha
  \hat\beta }^{\hat\gamma} \neq 0$, then
$\overline{g}_{\alpha+\hat\alpha, \beta+\hat\beta}^{\gamma+\hat\gamma}
\neq 0$. This implies that $\RKron_k$ is closed under
stretching. That is, that $\overline{g}_{\alpha, \beta}^{\gamma} \neq
0$ implies that $\overline{g}_{N\, \alpha, N\,\beta}^{N\,\gamma} \neq
0$ for all $N>0$.

 Both Klyachko and Kirillov have conjectured that the converse also
 holds. That is to say, that the reduced Kronecker coefficients satisfy the
 saturation property, \cite{Klyachko, Kirillov:saturation}. 
Remarkably, the Kronecker coefficients do not satisfy the saturation
 property. For example,
 \[
g^{(n,n)}_{(n,n),(n,n)}=0 \text{ if  $n$ is odd, but
  $g^{(n,n)}_{(n,n),(n,n)}=1$  if $n$  is even.}
\]
At this point, we hope that the reader is convinced that the reduced Kronecker
coefficients are  interesting objects on their own.

We are ready to describe the results of this article.
In Theorem \ref{theorem:g in gbar} we give an explicit formula for
recovering 
the value of the Kronecker coefficients  from the reduced Kronecker
coefficients.
Let $u=(u_1,u_2,\ldots)$ be an infinite sequence and $i$ a positive
integer. Define $u^{\dagger i}$ as the sequence obtained from $u$ by
adding $1$ to its $i-1$ first terms and erasing its $i$--th term:
\[
u^{\dagger i}=(1+u_1,1+u_2, \ldots, 1+u_{i-1}+1,u_{i+1}, u_{i+2}, \ldots)
\]
Partitions are identified with infinite sequences by appending
trailing zeros. Under this identification, when $\lambda$ is a
partition then so is $\lambda^{\dagger i}$ for all positive $i$.
\begin{theorem}[Computing the Kronecker coefficients from the reduced
    Kronecker coefficients]\label{theorem:g in gbar}
Let $n$ be a nonnegative integer and $\lambda$, $\mu$, and $\nu$ be
partitions of $n$.  Then
\begin{equation}\label{eq:g to gbar}
g_{\mu\nu}^\lambda
=
\sum_{i=1}^{\ell(\mu)\ell(\nu)} (-1)^{i+1} \bar{g}_{\bar{\mu}
\bar{\nu}}^{\lambda^{\dagger i}}
\end{equation}
\end{theorem}
This Theorem was  stated in  \cite{Briand:Orellana:Rosas:FPSAC}, and used  to compute  an explicit piecewise quasipolynomial description  for the Kronecker  coefficients indexed by two two--row shapes.
 

Murnaghan Theorem implies the stability property for the Kronecker products $s_{\alpha[n]} \ast s_{\beta[n]}$ presented in the introduction. Indeed, for $n$ big enough,
all sequences $\gamma[n]$ for $\gamma \in \Supp(\alpha,\beta)$ are partitions, and then \eqref{eq:thm Mur} is the expansion of $s_{\alpha[n]} \ast s_{\beta[n]}$ in the Schur basis. In particular, for $n$ big enough, the Kronecker coefficient $g_{\alpha[n],\beta[n]}^{\gamma[n]}$ is equal to the reduced Kronecker coefficient $\overline{g}_{\alpha,\beta}^{\gamma}$.

It is natural to ask about the index $n$ at which the expansion of $s_{\alpha[n]} \ast s_{\beta[n]}$ stabilizes. This index is defined as follows. 
\begin{defi}[$\stab(\alpha,\beta)$]
Let $V$ be the linear operator on symmetric functions defined on the Schur basis by: $V \left(s_{\lambda}\right)=s_{\lambda+(1)}$ for all partitions $\lambda$. Let $\alpha$ and $\beta$ be partitions. Then $\stab(\alpha,\beta)$ is defined as the smallest integer $n$ such that $s_{\alpha[n+k]} \ast s_{\beta[n+k]}=V^k \left(s_{\alpha[n]} \ast s_{\beta[n]}\right)$ for all $k>0$.
\end{defi}

As an illustration see the example in the introduction, there  $\alpha=\beta=(2)$
and the Kronecker
product is stable starting at $s_{(6,2)}\ast s_{(6,2)}$. Since $(6,2)$ is a
partition of $8$, we get that $\stab(\alpha,\beta)=8$. 

\begin{theorem}  \label{thm:global}
Let $\alpha$ and $\beta$ be two partitions. Then 
\[
\stab(\alpha,\beta) =|\alpha|+|\beta|+\alpha_1+\beta_1.
\]
\end{theorem}
In order to show that this theorem holds, we first reduce the calculation of $\stab(\alpha,\beta)$ to maximizing a linear form on $\Supp(\alpha,\beta)$ (Lemma \ref{lemma:stabbound}):
\[
\stab(\alpha,\beta) =\maxsupp{|\gamma|+\gamma_1}.
\]
Then,  we  show that (Theorem \ref{thm:g and g1})
\begin{equation}\label{eq:g and g1}
\maxsupp{|\gamma|+\gamma_1}
=|\alpha|+|\beta|+\alpha_1+\beta_1
\end{equation}
using a decomposition of $\overline{g}_{\alpha\beta}^\gamma$ as a sum of nonnegative summands
derived from Murnaghan's theorem, this decomposition is described in Lemma \ref{prop:Littlewood}.

We also obtain other interesting  bounds for linear forms over the set $\Supp(\alpha,\beta)$:
\begin{itemize}
\item In Theorem \ref{thm:gamma1} we show that:
\begin{equation}\label{eq:max gamma1}
\maxsupp{\gamma_1}
=
|\alpha \cap \beta|+\max(\alpha_1,\beta_1)
\end{equation}
\item More generally we obtain in Theorem \ref{prop:otherparts} that, 
whenever $\overline{g}_{\alpha,\beta}^{\gamma}>0$, we have for all positive integers $i$, $j$:
\[
\gamma_{i+j-1} \leq |E_i\alpha \cap E_j\beta|+\alpha_i+\beta_j
\]
where $E_k \lambda$ stands for the partition obtained from $\lambda$ by erasing its $k$--th part.
\item We also obtain (Theorem \ref{minmaxgamma}):
\begin{align*}
\maxsupp{|\gamma|}&= |\alpha|+|\beta|,\\
\minsupp{|\gamma|}&= \max(|\alpha|,|\beta|)-|\alpha\cap\beta|.
\end{align*}
\end{itemize}
Note that Formula \eqref{eq:max gamma1} is reminiscent 
to the following result for the Kronecker coefficients:
\begin{proposition}[ Klemm \cite{Klemm}, Dvir \cite{Dvir} Theorem 1.6, Clausen and Meier  \cite{Clausen:Meier} Satz 1.1.]\label{prop:Dvir}
Let $\alpha$ and $\beta$ be partitions with the same weight. Then:
\[
\max\left\{\gamma_1\;|\; \text{\rm $\gamma$ partition s. t. $g_{\alpha,\beta}^{\gamma}>0$} \right\}=|\alpha \cap \beta|
\]
\end{proposition}


In Section \ref{sec:coeff}, we consider  the weaker version of the stabilization
problem (Think uniform convergence vs simple convergence).  As
mentioned, Murnaghan's Theorem also implies that each particular
sequence of Kronecker coefficients
${g}_{\alpha[n],\beta[n]}^{\gamma[n]}$ stabilizes with value $\overline{g}_{\alpha,\beta}^{\gamma}$, possibly before
reaching $\stab(\alpha, \beta)$. More is known about these sequences:
     \begin{Britheorem}[Brion  \cite{Brion:Foulkes}, see also \cite{Manivel:rectangularKron}]\label{brion:monotone}
Let $\alpha$, $\beta$ and $\gamma$ be partitions. The sequence  with general term $g_{\alpha[n],\beta[n]}^{\gamma[n]}$ is weakly increasing.   \end{Britheorem}

The second stabilization problem consists in determining the following numbers.
\begin{defi}[$\stab(\alpha,\beta,\gamma)$]
Let $\alpha$, $\beta$, $\gamma$ be partitions. Then $\stab(\alpha,\beta,\gamma)$ is defined as the the smallest integer $N$
such that the sequences $\alpha[N]$, $\beta[N]$ and $\gamma[N]$ are partitions and
$
g_{{\alpha}[n],{\beta}[n]}^{{\gamma}[n]}=\overline{g}_{\alpha,\beta}^{\gamma}
$ for all  $n \geq N$.
\end{defi}

Lemma \ref{lemma:Mf} describes a general technique for producing linear upper bounds for $\stab(\alpha,\beta,\gamma)$ from any linear function $f$ such that $\gamma_1 \le f(\alpha, \beta, \bar \gamma)$ whenever $\overline{g}_{\alpha,\beta}^{\gamma}>0$. 
This method provides two new upper bounds $N_1$ and $N_2$ for
$\stab(\alpha,\beta, \gamma)$. 

The first bound is
found by applying Lemma  \ref{lemma:Mf} to the bound \eqref{eq:max gamma1} for
$\gamma_1$ obtained in Theorem \ref{thm:gamma1}. 
 \begin{theorem}\label{theorem:N1}
Let $M_1(\alpha,\beta;\gamma)=|\gamma|+|\bar{\alpha}\cap \bar{\beta}|+\alpha_1+\beta_1$ 
and 
\[
N_1(\alpha,\beta,\gamma)= \min \left\{
M_1(\alpha,\beta;\gamma), M_1(\alpha,\gamma;\beta), M_1(\beta,\gamma;\alpha) \right\}
\]
Then $\stab(\alpha,\beta,\gamma)\leq N_1(\alpha,\beta,\gamma)$.
\end{theorem}
The second bound is obtained by applying
Lemma \ref{lemma:Mf} to the bound \eqref{eq:g and g1} obtained in Theorem  \ref{thm:g and g1}.
\begin{theorem}\label{theorem:N2}
Let 
\[
N_2(\alpha,\beta,\gamma)
=\left[ 
\frac{|\alpha|+|\beta|+|\gamma|+\alpha_1+\beta_1+\gamma_1}{2}
\right]
\]
where $[x]$ denotes the integer part of $x$.
Then $\stab(\alpha,\beta,\gamma)\leq N_2(\alpha, \beta, \gamma)$.
\end{theorem}

We finish our work by placing the new bounds in the context of the current literature.
We show in Proposition \ref{prop:comparison} that $N_1$ beats those of Ernesto Vallejo \cite{Vallejo} and Michel Brion  \cite{Brion:Foulkes}. But neither one is better than the other since there are infinite families of examples  where $N_1 < N_2$ (see the Example \ref{ex:three hooks}  on the Kronecker coefficients indexed by three hooks), and others  where $N_2<N_1$ (see the Example \ref{ex:two two-row} on the Kronecker coefficients indexed by two two-row shapes). 
Finally, we revisit the work of   Rosas \cite{Rosas:2001}, Ballantine and Orellana \cite{Ballantine:Orellana}, and \cite{Briand:Orellana:Rosas:Chamber} where the situation for some restricted families of Kronecker coefficients  is addressed.




\section{The reduced Kronecker coefficients}\label{sec:RKC}



In this section we show how to recover the Kronecker coefficients from the knowledge of the reduced Kronecker coefficients. We also present an expression for the reduced Kronecker coefficients as sums of nonnegative terms, involving Littlewood--Richardson coefficients as well as Kronecker coefficients, that will be useful in the next two sections.



We denote by $\langle \ |\  \rangle$ the Hall inner product on symmetric functions. 
Recall that Formula \eqref{eq:g to gbar} in Theorem \ref{theorem:g in gbar} shows that we can recover the Kronecker coefficients from the reduced ones:
\[
g_{\mu\nu}^\lambda
=
\sum_{i=1}^{\ell(\mu)\ell(\nu)} (-1)^{i+1} \bar{g}_{\bar{\mu}
\bar{\nu}}^{\lambda^{\dagger i}}.
\]
 We now provide the proof.
\begin{proof}[Proof of Theorem \ref{theorem:g in gbar}] 
Murnaghan's theorem tells us that
\[
s_\mu\ast s_\nu=\sum_{\gamma \in \Supp(\bar{\mu},\bar{\nu})}
\bar{g}_{\bar{\mu}\bar{\nu}}^\gamma s_{\gamma[n]}
\]
Performing the scalar product with $s_{\lambda}$ in the preceding
equation yields:
\begin{equation}\label{eq:ggg}
g_{\mu,\nu}^{\lambda}=\sum_{\gamma \in \Supp(\bar{\mu},\bar{\nu})}
\bar{g}_{\bar{\mu}\bar{\nu}}^\gamma
\scalar{s_{\gamma[n]}}{s_{\lambda}}
\end{equation}

Consider a particular $\gamma \in \Supp(\bar{\mu},\bar{\nu})$ such
that $\scalar{s_{\gamma[n]}}{s_{\lambda}}\neq 0$. Let $k$ be its
length. Then $\lambda$ has length at most $k+1$ and  the Jacobi--Trudi
determinants $s_{\gamma[n]}$ and $s_{\lambda}$ have the same columns,
up to the order, see Eq. \eqref{eq:JT}. That is, the sequence 
\[
v = (n-|\gamma|,
\gamma_1,\gamma_2,\ldots,\gamma_k)+(k+1,k,k-1,\ldots,1)
\]
is a
permutation of the decreasing sequence 
$u = \lambda +
(k+1,k,k-1,\ldots,1)$. (As usual one sets $\lambda_j=0$ for $j>
\ell(\lambda)$.)

By construction, we have that  $v$ is decreasing starting at $v_2$. Therefore, there
exists an index $i$ such that $u_j = v_j +1$ for all $j < i$ and $u_j
= v_j$ for all $j > i$. This means that $\gamma = \lambda^{\dagger i}$
for some $i \leq k+1$. Since $\gamma \in \Supp(\bar{\mu},\bar{\nu})$
there is $k \leq \ell_1 \ell_2-1$ and thus $i \leq k$. 

Finally
$\scalar{s_{\gamma[n]}}{s_{\lambda}}$ is the sign of the permutation
that transforms  $v$ into the decreasing sequence $u$. This
permutation is the cycle $(i, i-1,\ldots, 2, 1)$, which has sign
$(-1)^{i+1}$.
This shows that only the partitions $\gamma=\lambda^{\dagger i}$, for
$i$ between $1$ and $\ell_1\ell_2$, contribute to the sum in the
right--hand side of \eqref{eq:ggg}, and that the contribution of
$\lambda^{\dagger i}$ is $(-1)^{i+1} \bar{g}_{\bar{\mu}\bar{\nu}}^\gamma$.

\end{proof}

The operator on symmetric functions  $f \mapsto f^\perp$ is defined  as the operator dual to multiplication with respect to the inner product, $\langle \, |\, \rangle$. 

Define $c_{\alpha,\beta,\gamma}^\delta$ as the coefficients of
$s_{\delta}$ in $s_{\alpha}s_\beta s_\gamma$. From the definition of
the Littlewood--Richardson coefficients as the structural constant for
the product of two Schur functions, we immediately obtain that
\begin{equation}\label{3LR}
c_{\alpha,\beta,\gamma}^\delta = \sum_\varphi c_{\alpha,\beta}^\varphi c_{\varphi,\gamma}^\delta
\end{equation}

\begin{lemma}\label{prop:Littlewood}
Let $\alpha$, $\beta$, $\gamma$ be partitions. Then $\overline{g}_{\alpha,\beta}^{\gamma}$ is positive if and only if there exist partitions $\delta$, $\epsilon$, $\zeta$, $\rho$, $\sigma$, $\tau$ such that all four coefficients $g_{\delta,\epsilon}^{\zeta}$,  $c_{\delta,\sigma,\tau}^{\alpha}$,  $c_{\epsilon,\rho,\tau}^{\beta}$ and  $c_{\zeta,\rho,\sigma}^{\gamma}$ are positive. Moreover,
\begin{align}\label{trick}
\overline{g}_{\alpha,\beta}^{\gamma}=
\sum g_{\delta,\epsilon}^{\zeta} c_{\delta,\sigma,\tau}^{\alpha} c_{\epsilon,\rho,\tau}^{\beta} c_{\zeta,\rho,\sigma}^{\gamma}
\end{align}

\end{lemma}


\begin{proof}
Given partitions $\alpha$ and $\beta$, define the following symmetric function 
\[
R_{\alpha,\beta}=\sum_{\delta,\epsilon,\tau} 
\left( (s_{\delta} s_{\tau})^{\perp} s_{\alpha}\right)
\left( (s_{\epsilon} s_{\tau})^{\perp} s_{\beta}\right)
\left( s_{\delta} \ast s_{\epsilon} \right)
\]
where the sum is over all triples of partitions $\delta$, $\epsilon$, $\tau$. For $n$ integer, let $U_n$ be the linear operator on symmetric functions that sends the Schur function $s_{\lambda}$ to the Jacobi--Trudi determinant $s_{\lambda[n]}$. 
Littlewood showed in \cite{Littlewood:1958} that for all partitions $\alpha$ and $\beta$ and all integers $n$, 
\begin{equation}\label{eq:Littlewood}
s_{\alpha[n]} \ast s_{\beta[n]}=U_n R_{\alpha,\beta}
\end{equation}
Formula \eqref{eq:Littlewood} is also presented in \cite{Butler:King} (Formula 6.1) and \cite{Scharf:Thibon:Wybourne} (Formula 8). 

Comparing \eqref{eq:Littlewood} with Murnaghan's Theorem we see that $U_n R_{\alpha,\beta}=U_n \sum_{\gamma} \overline{g}_{\alpha,\beta}^{\gamma} s_{\gamma}$. The operator $U_n$ is not injective, but its restriction to the symmetric functions of degree at most $n/2$ is. Indeed, when $|\gamma| \leq n/2$, the sequence $\gamma[n]$ is a partition. Therefore, taking $n$ big enough we can deduce that $R_{\alpha,\beta}=\sum_{\gamma} \overline{g}_{\alpha,\beta}^{\gamma} s_{\gamma}$. 

Let us determine the expansion $\sum_{\gamma} r_{\alpha,\beta}^{\gamma} s_{\gamma}$ of $R_{\alpha,\beta}$ in the Schur basis. We have:
\begin{align*}
(s_{\delta} s_{\tau})^{\perp}s_{\alpha}&=\sum_{\sigma} c_{\delta,\sigma,\tau}^{\alpha} s_{\sigma},\\
(s_{\epsilon} s_{\tau})^{\perp}s_{\beta}&=\sum_{\rho} c_{\epsilon,\rho,\tau}^{\beta} s_{\rho},\\
s_{\delta} \ast s_{\epsilon}&=\sum_{\zeta} g^{\zeta}_{\delta,\epsilon} s_{\zeta}
\end{align*}
Therefore, 
\begin{align*}
R_{\alpha,\beta} &=
\sum g^{\zeta}_{\delta,\epsilon} c_{\delta,\sigma,\tau}^{\alpha} c_{\epsilon,\rho,\tau}^{\beta}  s_{\sigma} s_{\rho} s_{\tau}\\
&=
\sum g^{\zeta}_{\delta,\epsilon} c_{\delta,\sigma,\tau}^{\alpha} c_{\epsilon,\rho,\tau}^{\beta} c_{\sigma,\rho,\tau}^{\gamma} s_{\gamma}
\end{align*}
We obtain Eq. (\ref{trick}).
\end{proof}

\section{Stability : The Kronecker product}\label{sec:product}


In this section we consider  the stability of the Kronecker product of Schur functions. We provide a proof for Theorem \ref{thm:global} which provides a sharp bound for this stability. 

\begin{lemma}
\label{lemma:stabbound} 
Let $\alpha$ and $\beta$ be partitions. Then 
\[
\stab(\alpha,\beta) =\maxsupp{|\gamma|+\gamma_1}
\]
\end{lemma}

\begin{proof}
Let $N=\maxsupp{|\gamma|+\gamma_1}$. If $\alpha$ and $\beta$ are equal to the empty partition then $N=0=\stab(\alpha,\beta)$. In the other cases, that we consider now, we have $N>0$.

Remember (from the definition of $\stab(\alpha,\beta)$ in Section \ref{sec:main results}) that $V$ is the linear operator that fulfills $V\left(s_{\lambda}\right)=s_{\lambda+(1)}$ for all partitions $\lambda$. For all $\gamma \in \Supp(\alpha,\beta)$ and $k > 0$, the sequences $\gamma[N]$ and $\gamma[N+k]$ are  partitions, therefore $s_{\gamma[N+k]}=V^k\left(s_{\gamma[N]}\right)$. After Murnaghan's Theorem,
\begin{align*}
s_{\alpha[N]}\ast s_{\beta[N]} &=\sum_{\gamma \in \Supp(\alpha,\beta)} \overline{g}_{\alpha\beta}^\gamma s_{\gamma[N]},\\
s_{\alpha[N+k]}\ast s_{\beta[N+k]} &=\sum_{\gamma \in \Supp(\alpha,\beta)} \overline{g}_{\alpha\beta}^\gamma s_{\gamma[N+k]}.
\end{align*}
We obtain that:
\[
s_{\alpha[N+k]}\ast s_{\beta[N+k]}=V^k \left( s_{\alpha[N]}\ast s_{\beta[N]} \right).
\]
This proves that $N \geq \stab(\alpha,\beta)$.

The equality will be obtained by proving additionally that $N-1 < \stab(\alpha,\beta)$.
There exists a partition $\gamma \in \Supp(\alpha,\beta)$ such that $|\gamma|+\gamma_1=N$. Then $\gamma[N]$ is a partition with its first part equal to its second part. This shows that $s_{\gamma[N]}$ is not in the image of $V$. It follows that $s_{\alpha[N]} \ast s_{\beta[N]}$ is not in the image of $V$. In particular,  $s_{\alpha[N]} \ast s_{\beta[N]}$ is not equal to $V \left(s_{\alpha[N-1]} \ast s_{\beta[N-1]}\right)$. 
\end{proof}



\begin{theorem}\label{thm:g and g1}
Let $\alpha$, $\beta$ be partitions. Then,
\begin{equation}\tag{\ref{eq:g and g1}}
\maxsupp{|\gamma|+\gamma_1}=|\alpha|+|\beta|+\alpha_1+\beta_1.
\end{equation}
\end{theorem}

\begin{proof}
Let $\gamma$ be a partition such that $\overline{g}_{\alpha,\beta}^{\gamma}>0$. By Lemma \ref{prop:Littlewood}, there exist partitions $\delta$, $\epsilon$, $\zeta$, $\rho$, $\sigma$, $\tau$ such that all four coefficients $g_{\delta,\epsilon}^{\zeta}$,  $c_{\delta,\sigma,\tau}^{\alpha}$,  $c_{\epsilon,\rho,\tau}^{\beta}$ and  $c_{\zeta,\rho,\sigma}^{\gamma}$ are positive. 

The Littlewood--Richardson rule together with Eq. \eqref{3LR} implies that if $c_{\zeta,\rho,\sigma}^{\gamma}>0$ then $\gamma_1 \leq \zeta_1+\rho_1+\sigma_1$. 
Since $c_{\zeta,\rho,\sigma}^{\gamma}>0$, we have also $|\gamma|=|\zeta|+|\rho|+|\sigma|$. Therefore $|\gamma|+\gamma_1 \leq |\zeta|+\zeta_1+|\rho|+\rho_1+|\sigma|+\sigma_1$. Obviously $\zeta_1 \leq |\zeta|$. Thus 
\begin{equation}\label{eq:12}
|\gamma|+\gamma_1 \leq 2\,|\zeta|+|\rho|+\rho_1+|\sigma|+\sigma_1
\end{equation}

Since $g_{\delta,\epsilon}^{\zeta}>0$ we have $|\zeta|=|\delta|=|\epsilon|$. Replacing $2|\zeta|$ with $|\delta|+|\epsilon|$ in \eqref{eq:12} yields 
\begin{equation}\label{eq:13}
|\gamma|+\gamma_1 \leq |\delta|+|\sigma|+\sigma_1+|\epsilon|+|\rho|+\rho_1
\end{equation}
Since  $c_{\delta,\sigma,\tau}^{\alpha}>0$ we have  $\sigma \subset \alpha$ and thus $\sigma_1 \leq \alpha_1$. We have also $|\delta|+|\sigma|\leq |\alpha|$. Therefore  $|\delta|+|\sigma|+\sigma_1 \leq |\alpha|+\alpha_1$. 

Similarly, $c_{\epsilon,\rho,\tau}^{\beta}>0$ implies $|\epsilon|+|\rho|+\rho_1 \leq |\beta|+\beta_1$. 

Substituting these two new inequalities in \eqref{eq:13} provides the following inequality
\[
|\gamma|+\gamma_1 \leq |\alpha|+|\beta|+\alpha_1+\beta_1.
\]

We now show that the bound is achieved. 
Consider the reduced Kronecker coefficient $\overline{g}_{\alpha,\beta}^{\alpha+\beta}$. The Murnaghan--Littlewood theorem implies that it is equal to the Littlewood--Richardson coefficient $c_{\alpha,\beta}^{\alpha+\beta}$ which is equal to $1$. This proves that the upper bound $|\alpha|+|\beta|+\alpha_1+\beta_1$ on $\Supp(\alpha,\beta)$, for $|\gamma|+\gamma_1$, is reached with $\gamma=\alpha+\beta$.
\end{proof}

Theorem \ref{thm:global} is now a direct consequence of Lemma \ref{lemma:stabbound} and Theorem \ref{thm:g and g1}.


\section{Bounds for linear forms on $\Supp(\alpha,\beta)$}\label{sec:bounds}

In this section we provide proofs for the bounds of the lengths of the rows of $\gamma$ when $\overline{g}_{\alpha,\beta}^\gamma >0$.  In particular, we provide a sharp bound for the first row and upper bounds for the remaining rows.  Theorem \ref{thm:gamma1} gives a first step towards describing the set partitions indexing  the nonzero reduced Kronecker coefficients, that is $\mbox{Supp}(\alpha,\beta)$. 

Indeed, we show that

\begin{theorem}\label{thm:gamma1} 

Let $\alpha$ and $\beta$ be partitions,  then
\[
\maxsupp{\gamma_1}
=
|\alpha \cap \beta|+\max(\alpha_1,\beta_1)
\]
\end{theorem}

From Theorem \ref{thm:gamma1}, we obtain that given  any three partitions $\mu, \nu$ and $\lambda$ of
  $n$. If $g_{\mu,\nu}^\lambda>0$ then 
\[
\lambda_2\leq \min(\frac{n}{2}, |\bar{\mu}\cap \bar{\nu}|
+\max(\mu_2,\nu_2)).
\]

Fix two partitions $\alpha$ and $\beta$.  To prove Theorem \ref{thm:gamma1} we first prove an upper bound for all the rows of $\gamma$ whenever $\overline{g}_{\alpha,\beta}^\gamma>0$ (Theorem \ref{prop:otherparts}).  

For $\lambda$ partition and $k$ positive integer, set $E_k \lambda$ for the partition obtained from $\lambda$ by erasing its $k$--th part (or leaving $\lambda$ unchanged when it has less than $k$ parts). In particular $E_1 \lambda=\cut{\lambda}$.

\begin{lemma}\label{lemma:H}
Let $\alpha$, $\delta$, $\sigma$ and $\tau$ be partitions such that $c_{\delta,\sigma,\tau}^{\alpha} >0$. Let $i$ be a positive integer. Then there exists a set $A$ such that $D(\delta) \subset D(E_i \alpha) \cup A$ and $|A|+\sigma_k \leq \alpha_k$.
\end{lemma}

\begin{proof}
By Eq. \eqref{3LR}, there exists a partition $\kappa$ such that $c_{\kappa,\tau}^{\alpha}>0$ and $c_{\delta,\sigma}^{\kappa}>0$ since $c_{\delta, \sigma,\tau}^\alpha>0$. In particular $D(\delta) \subset D(\kappa) \subset D(\alpha)$.

Let $S_i=\{(x,y)\,|\,x \geq 1 \text{ and }y \geq i\}$ and let $H=D(\delta)\setminus D(\cut{\kappa})$. Notice that $H$  is an horizontal strip consisting in all boxes of $D(\delta)$ having no box of $D(\kappa)$ above them, see Figure \ref{fig:strip} for an example.
\begin{figure}[h]
\includegraphics{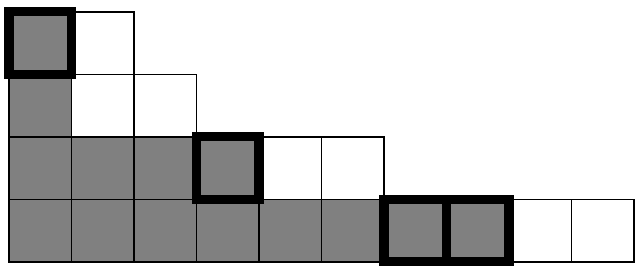}
\caption{The horizontal strip $H$ (boxes with thick edges)  
for $\kappa=(10,6,3,2)$ (white and grey boxes) and $\delta=(8,4,1,1)$ (grey boxes). 
}\label{fig:strip}
\end{figure}

Let $A=S_i \cap H$, notice that this is the horizontal strip contained in $H$ strictly above the $i-1$-st row. We have
\[
|A|=\kappa_i - \width( D(\kappa/\delta) \cap S_i ).
\]
On the other hand, since $c_{\delta,\sigma}^{\kappa}>0$, there exists a Littlewood--Richardson tableau with shape $\kappa/\delta$ and content $\sigma$. In this tableau, there is at most one occurrence of $i$ by column of $\kappa/\delta$, and they are all in row $i$ or higher. Therefore,
\[
\sigma_i \leq \width(D(\kappa/\delta) \cap S_i).
\]
As a consequence,
\[
|A| +\sigma_i \leq \kappa_i 
\]
Since $D(\kappa) \subset D(\alpha)$ we conclude that $|A| +\sigma_i \leq \alpha_i$.

Now by construction of $A$,
\[
D(\delta) \cap S_i \subset \left(D(\cut{\kappa}) \cap S_i\right) \cup A
\]
and clearly $D(\delta) \setminus S_i \subset D(\kappa) \setminus S_i$. Therefore
\[
D(\delta) \subset \left( D(\kappa) \setminus S_i\right) \cup \left( D(\cut{\kappa}) \cap S_i\right) \cup A
\]
Finally, observe that $D(E_i \kappa)=\left( D(\kappa) \setminus S_i\right) \cup \left( D(\cut{\kappa}) \cap S_i\right)$. Therefore, 
\[
D(\delta) \subset D(E_i \kappa) \cup A
\]
Since $D(\kappa) \subset D(\alpha)$ we have $D(E_i \kappa) \subset D(E_i \alpha)$, and thus
\[
D(\delta) \subset D(E_i \alpha) \cup A
\]
\end{proof}

\begin{theorem}\label{prop:otherparts}
Let $\alpha$, $\beta$ and $\gamma$ be partitions such that $\overline{g}_{\alpha,\beta}^{\gamma}>0$ and let $i,j$ and $k$ be a positive integers such that $i+j-1=k$, then we have
\[
\gamma_k \leq |E_i \alpha \cap E_j \beta | +\alpha_i + \beta_j.
\]
\end{theorem}

\begin{proof}
Let $i$ and $j$ such that $i+j-1=k$.

By Lemma \ref{prop:Littlewood}, there exist partitions $\delta$, $\epsilon$, $\zeta$, $\rho$, $\sigma$, $\tau$ such that all four coefficients $g_{\delta,\epsilon}^{\zeta}$, $c_{\delta,\sigma,\tau}^{\alpha}$, $c_{\epsilon,\rho,\tau}^{\beta}$, $c_{\zeta,\rho,\sigma}^{\gamma}$ are positive. 

By Eq. \eqref{3LR} and since $c_{\zeta,\rho,\sigma}^{\gamma}>0$, there exists a partition $\phi$ such that $c_{\zeta,\phi}^{\gamma}>0$ and $c_{\rho,\sigma}^{\phi}>0$. Weyl's inequalities for eigenvalues of Hermitian matrices (\cite{Weyl} or Eq. (2) in \cite{Fulton}) imply that whenever a Littlewood--Richardson coefficient $c_{\mu,\nu}^{\lambda}$ is non--zero there is $\lambda_{p+q-1} \leq \mu_p + \nu_q$ for all $p$, $q$ (see \cite{Fulton}). Apply this to $c_{\zeta,\phi}^{\gamma}$ with $p=1$, $q=k$: we obtain $\gamma_k \leq \zeta_1+\phi_k$. Apply Weyl's inequalities to $c_{\rho,\sigma}^{\phi}$ with $p=j$, $q=i$: we obtain $\phi_k \leq \rho_j+\sigma_i$. It follows that $\gamma_k \leq \zeta_1+\sigma_i+\rho_j$.

Since $g_{\delta,\epsilon}^{\zeta}>0$, we have $\zeta_1 \leq |\delta \cap \epsilon|$ by Proposition \ref{prop:Dvir}, then
\begin{equation}\label{eq:11}
\gamma_k \leq |\delta \cap \epsilon|+\rho_j+\sigma_i
\end{equation}
Since $c_{\delta,\sigma,\tau}^{\alpha}>0$, Lemma \ref{lemma:H} implies that there exists a set $A_1$ such that
\[
D(\delta) \subset D(E_i \alpha) \cup A_1 \text{ and } |A_1| + \sigma_i \leq \alpha_i.
\]
Similarly for $c_{\epsilon,\rho,\tau}^{\beta}>0$,  Lemma \ref{lemma:H} implies that there exists a set $A_2$ such that
\[
D(\epsilon) \subset D(E_j \beta) \cup A_2 \text{ and } |A_2| + \rho_j \leq \beta_j.
\]
Therefore,
\[
D(\delta \cap \epsilon) \subset D(E_i \alpha \cap E_j \beta) \cup A_1 \cup A_2.
\]
As a consequence,
\[
|\delta \cap \epsilon| \leq |E_i \alpha \cap E_j \beta| + |A_1| + |A_2|.
\]
This together with \eqref{eq:11} yields
\[
\gamma_k \leq |E_i \alpha \cap E_j \beta| + |A_1| + \sigma_i+|A_2|+\rho_j.
\]
Remembering that $|A_1| + \sigma_i \leq \alpha_i$ and $|A_2| + \rho_j \leq \beta_j$, we get the claimed inequality.
\end{proof}

\begin{proof}[Proof of Theorem \ref{thm:gamma1}]
The bound holds by Theorem \ref{prop:otherparts} since $|E_1\alpha\cap E_1\beta|+\alpha_1+\beta_1=|\alpha\cap \beta|+\max(\alpha_1,\beta_1)$.
 Let us now show it is reached. Choose $\delta=\epsilon=\cut{\alpha} \cap \cut{\beta}$ and for $\zeta$ a partition such that $g_{\delta,\epsilon}^{\zeta}>0$ and $\zeta_1=|\delta \cap \epsilon|=|\cut{\alpha} \cap \cut{\beta}|$, such a partition exists by Proposition \ref{prop:Dvir}. Choose $\tau$ to be the empty partition. 

Choose $\sigma$ as follows: first, $\sigma_1=\alpha_1$. This will ensure that $c_{\delta,\sigma,\tau}^{\alpha}=c_{\delta,\sigma}^{\alpha}=c_{\delta,\cut{\sigma}}^{\cut{\alpha}}$. 
The Littlewood--Richardson coefficients $c_{\delta,\kappa}^{\cut{\alpha}}$ are the coefficients in the expansion of the non--zero skew--Schur function $s_{\cut{\alpha}/\delta}$ in the Schur basis, hence one of them has to be non--zero. Choose for $\cut{\sigma}$ one such partition $\kappa$ (observe that $D(\kappa) \subset D(\cut{\alpha})$, therefore $\kappa_1 \leq \alpha_2 \leq \alpha_1=\sigma_1$).  Define similarly $\rho$. Finally set $\gamma=\zeta+\sigma+\rho$. 
\end{proof}

\begin{theorem}[The maximum and minimum weight of partitions indexing nonzero reduced
    Kronecker coefficients]\label{minmaxgamma}
Let $\alpha$ and $\beta$ be partitions. We have:
\begin{align*}
\maxsupp{|\gamma|} &= |\alpha|+|\beta|,\\
\minsupp{|\gamma|} &= \max(|\alpha|,|\beta|)-|\alpha\cap\beta|.
\end{align*}
\end{theorem}
\proof
From Murnaghan's inequalities we know that 
$|\gamma|\le|\alpha|+|\beta|$ for all $\gamma \in \Supp(\alpha,\beta)$. Moreover, this maximum is achieved, take $\gamma=\alpha+\beta$, then $c_{\alpha,\beta}^{\alpha+\beta}>0$ and finally $\overline{g}_{\alpha,\beta}^{\alpha+\beta}=c_{\alpha,\beta}^{\alpha+\beta}$ by the theorem of Littlewood and Murnaghan.

To show the second bound, assume that $\overline{g}_{\alpha,\beta}^{\gamma}>0$. There exists $n$ such that 
$g_{\alpha[n],\beta[n]}^{\gamma[n]}=\overline{g}_{\alpha,\beta}^{\gamma}$. By Proposition \ref{prop:Dvir} we have that 
$
n-|\gamma|  \leq |\alpha[n]\cap \beta[n]|.
$
Hence,  
\[
|\alpha[n]\cap \beta[n]| = \min(n-|\alpha|,n-|\beta|)+|\alpha \cap \beta|
                         = n-\max(|\alpha|,|\beta|)+|\alpha \cap \beta|.
\]
We conclude that $|\gamma| \geq \max(|\alpha|, |\beta|) -|\alpha\cap\beta|$.  

Again by Proposition \ref{prop:Dvir} we know that there is a partition $\gamma$ for which $n-|\gamma| = |\alpha[n]\cap\beta[n]|$, hence this bound is sharp. 
\endproof

\begin{corollary}
\label{thm:globalboundlower}
Let $\alpha$ and $\beta$ be partitions and $i$ and $j$ positive integers such that $k=i+j-1$. Then
\[
\maxsupp{\gamma_k}
\le 
\min\left( |E_i \alpha \cap E_j \beta | +\alpha_i + \beta_j,\left[\frac{ |\alpha|+|\beta|}{k}\right]  \right)
\]

\end{corollary}

\begin{proof}  This is a straightforward consequence of Theorems \ref{prop:otherparts} and   \ref{minmaxgamma}.
\end{proof}

\begin{exa} \label{ex:otherparts}
Let $\alpha=(2)$ and $\beta=(4,3,2)$, then the first row of the table are the nonzero values of $\gamma_k$ and the second row are the upper bounds given by Corollary \ref{thm:globalboundlower}.
\begin{center}
\begin{tabular}{l|c|c|c|c|c}
$k$ & 1 & 2 & 3 &4 &5 \\ \hline
max values for $\gamma_k$ & 6&4&3&2&1\\ \hline
bound for $\gamma_k$ &6&5&3&2&2
\end{tabular}
\end{center}
In the case that $\alpha=(3,1)$ and $\beta=(2,2)$ we get
\begin{center}
\begin{tabular}{l|c|c|c|c|c|c}
$k$ & 1 & 2 & 3 &4 &5& 6 \\ \hline
max values for $\gamma_k$ & 6&3&2&1&1&1\\ \hline
bound for $\gamma_k$ & 6 & 4 & 2 & 2&1 &1
\end{tabular}
\end{center}
\end{exa}

\section{Stability : The Kronecker coefficients}\label{sec:coeff}


In this last section we consider linear upper bounds for $\stab(\alpha,\beta,\gamma)$. Previously known bounds, due to Brion   \cite{Brion:Foulkes} and Vallejo \cite{Vallejo} respectively, are
\begin{align*}
M_B(\alpha,\beta;\gamma)&=|\alpha|+|\beta|+\gamma_1,\\
M_V(\alpha,\beta; \gamma)&=|\gamma|+\left\lbrace
\begin{array}{ll}
\max \{|\alpha|+\alpha_1-1,|\beta|+\beta_1-1,|\gamma|\} & \text{ if } \alpha \neq \beta \\
\max \{|\alpha|+\alpha_1,|\gamma|\} & \text{ if } \alpha=\beta
\end{array}
\right.
\end{align*}

We introduce Lemma \ref{lemma:Mf} that produces linear upper bounds for
$\stab(\alpha, \beta, \gamma)$ from linear inequalities fulfilled by those $(\alpha,\beta,\gamma)$ 
 for which $\overline{g}_{\alpha,\beta}^{\gamma} >0.$ 
Applying this lemma to different bounds derived in Section 3 and Section 4, we obtain two new upper
 bounds for $\stab(\alpha, \beta, \gamma)$, and recover Brion's bound $M_B$.

\begin{lemma}\label{lemma:Mf}
Let $f$ be a function on triples of partitions such that for all $i$,
\[
f(\alpha,\beta,\bar{\gamma}) \geq f(\alpha,\beta,\gamma^{\dagger i}).
\]
Set $\mathcal{M}_f(\alpha,\beta,\gamma)=|\gamma|+f(\alpha,\beta,\bar{\gamma})$ and assume also that 
whenever $\overline{g}_{\alpha,\beta}^{\gamma} >0$,
\begin{equation}\label{eq:condition}
\mathcal{M}_f(\alpha,\beta,\gamma) \geq \max\left(|\alpha|+\alpha_1,|\beta|+\beta_1,|\gamma|+\gamma_1\right).
\end{equation}
Then whenever $\overline{g}_{\alpha,\beta}^{\gamma}>0$,
\[
\stab(\alpha,\beta,\gamma) \leq \mathcal{M}_f(\alpha,\beta,\gamma).
\]
\end{lemma}

\begin{proof}
Let $\alpha$, $\beta$ and $\gamma$ be partitions such that $\overline{g}_{\alpha,\beta}^{\gamma}>0$. Let $n \geq \mathcal{M}_f(\alpha,\beta,\gamma)$.
By Lemma \ref{theorem:g in gbar},
\begin{equation}\label{eq:g in gbar with gamma}
g_{\alpha[n]\beta[n]}^{\gamma[n]} = 
\overline{g}_{\alpha,\beta}^{\gamma}+
\sum_{i=1}^{N}(-1)^{i}\,
\overline{g}_{\alpha,\beta}^{(n-|\gamma|+1,\gamma^{\dagger i})}
\end{equation}
for some $N$. Since $n \geq \mathcal{M}_f(\alpha,\beta,\gamma)=|\gamma|+f(\alpha,\beta,\bar{\gamma})$, we have 
$n-|\gamma|+1 > f(\alpha,\beta,\bar{\gamma})$. Thus $n-|\gamma|+1 > f(\alpha,\beta,\gamma^{\dagger i})$ for all $i$. As a consequence, none of the partitions $\tau=(n-|\gamma|+1,\gamma^{\dagger i})$ fulfills $\mathcal{M}_f(\alpha,\beta,\tau) \geq |\tau|+\tau_1$. Indeed, for such a partition, $|\tau|+\tau_1=|\tau|+(n-|\gamma|+1)$ and $\mathcal{M}_f(\alpha,\beta,\tau)=|\tau|+f(\alpha,\beta,\gamma^{\dagger i})$. We get that all terms $\overline{g}_{\alpha,\beta}^{(n-|\gamma|+1,\gamma^{\dagger i})}$ in \eqref{eq:g in gbar with gamma} are zero. Therefore $g_{\alpha[n]\beta[n]}^{\gamma[n]}$ is equal to its stable value $\overline{g}_{\alpha,\beta}^{\gamma}$. We conclude that $\mathcal{M}_f \geq \stab(\alpha,\beta,\gamma)$.
\end{proof}

Three functions $f$ such that \eqref{eq:condition} hold have already appeared in this paper. Each one gives a bound for $\stab(\alpha,\beta,\gamma)$.
\begin{enumerate}
\item Murnaghan's triangle inequalities (see Murnaghan's Theorem) and Theorem \ref{thm:gamma1} show that \eqref{eq:condition} holds for $f(\alpha,\beta,\tau)=|\alpha|+|\beta|-|\tau|$. We recover Brion's bound $M_B$.
\item Theorem \ref{thm:gamma1} and Murnaghan's triangle inequalities also imply that \eqref{eq:condition} holds for $f(\alpha,\beta,\tau)=|\bar{\alpha} \cap \bar{\beta}|+\alpha_1+\beta_1$.
the corresponding bound $\mathcal{M}_f$ is $M_1(\alpha,\beta,\gamma)=|\gamma|+|\bar{\alpha}\cap \bar{\beta}|+\alpha_1+\beta_1$. Hence, by Lemma \ref{lemma:Mf} and the symmetry of the Kronecker coefficients we obtain the proof of Theorem \ref{theorem:N1}.
\item Theorem \ref{thm:global} shows that \eqref{eq:condition} holds for $f(\alpha,\beta,\tau)=1/2\,(|\alpha|+|\beta|+\alpha_1+\beta_1-|\tau|)$, which corresponds to $\mathcal{M}_f=M_2=\frac{1}{2}(|\alpha|+|\beta|+|\gamma|+\alpha_1+\beta_1+\gamma_1)$. The bound $N_2=\left[ M_2 \right]$ of Theorem \ref{theorem:N2} follows. 
\end{enumerate}
Set $N_1(\alpha,\beta,\gamma)=\min\left\{
M_1(\alpha,\beta;\gamma), M_1(\alpha,\gamma;\beta), M_1(\beta,\gamma;\alpha) \right\}$ and define similarly $N_B$ and $N_V$ from $M_B$ and $M_V$. These are also upper bounds for $\stab(\alpha,\beta,\gamma)$. In the following proposition we show that the bound $N_1$ improves
both Vallejo's $N_V$ and Brion's bound, $N_B$. 
%

\begin{proposition} \label{prop:comparison}
Let $\alpha$, $\beta$, $\gamma$ be partitions, then $N_1(\alpha,\beta,\gamma) \leq N_B(\alpha,\gamma,\beta)$ and $N_1(\alpha,\beta,\gamma) \leq N_V(\alpha,\beta,\gamma)$.
\end{proposition}
\begin{proof} 
For all partitions $\alpha$, $\beta$, $\gamma$, we have
\[
M_1(\alpha,\beta;\gamma)=|\gamma|+|\bar{\alpha}\cap \bar{\beta}|+\alpha_1+\beta_1 \leq
|\gamma|+|\alpha|+\beta_1=M_B(\alpha,\gamma;\beta),
\]
since $|\bar{\alpha}\cap
\bar{\beta}|+\alpha_1 \leq |\bar{\alpha}|+\alpha_1 =|\alpha|$. This is enough to conclude that
 $N_1(\alpha,\beta,\gamma) \leq N_B(\alpha,\beta,\gamma)$.

We now prove that $N_1(\alpha,\beta,\gamma) \leq N_V(\alpha,\beta,\gamma)$. It is enough to prove that for all partitions $\alpha$, $\beta$, $\gamma$ we have $M_1(\alpha,\beta;\gamma) \leq M_V(\alpha,\beta;\gamma)$. By symmetry of both bounds with respect to $\alpha$ and $\beta$, we can assume without loss of generality that $\alpha_1 \geq \beta_1$. We consider three cases: $\alpha = \beta$; $\alpha \subsetneq \beta$; $\alpha \not \subset \beta$. We show that in the first case $|\bar{\alpha}\cap \bar{\beta}|+\alpha_1+\beta_1 \leq |\alpha|+\alpha_1$ and that in the other two cases $|\bar{\alpha}\cap \bar{\beta}|+\alpha_1+\beta_1 \leq \max(|\alpha|+\alpha_1-1,|\beta|+\beta_1-1)$.

Consider the case $\alpha = \beta$. Then $|\bar{\alpha}\cap \bar{\beta}|+\alpha_1+\beta_1=|\alpha|+\alpha_1$. 

Consider now the case $\alpha \subsetneq \beta$. Then $|\bar{\alpha}\cap \bar{\beta}|+\alpha_1=|\alpha|\leq |\beta|-1$. Therefore $|\bar{\alpha}\cap \bar{\beta}|+\alpha_1+\beta_1 \leq |\beta|+\beta_1-1$.

Consider last the case when $\alpha \not \subset \beta$. There is $|\bar{\alpha}\cap \bar{\beta}|+\beta_1=|\alpha\cap \beta| \leq |\alpha|-1$. Therefore $|\bar{\alpha}\cap \bar{\beta}|+\alpha_1+\beta_1 \leq |\alpha|+\alpha_1-1$.
\end{proof}

Now that we have shown that $N_1$ improves the bounds $N_B$ and $N_V$. In the following two examples we now compare $N_2$  to $N_B$ and $N_V$. 

\begin{exa}[Comparison of $N_2$ to $N_B$]
Let $\alpha = (2,1)$ and $\beta=(3,1)$, if $\gamma =(3,1)$, then $N_B=10$ is greater than $N_2=9$ and if $\gamma=(3,2,2)$ then $N_B=10$ and $N_2=11$.  This shows that neither one is better than the other. 
\end{exa}

\begin{exa}[Comparison of $N_2$ to $N_V$]
Let $\alpha=(2,1)$, $\beta=(3,1)$ and $\gamma=(3,2,2)$, then $N_2=11$ and $N_V=12$, hence $N_2<N_V$.   On the other hand if $\alpha=(3,2)$ and $\beta=(3,1,1)$ and $\gamma=(6)$, then $N_V=13$ and $N_2=14$ and in this case, $N_V<N_2$. This shows that neither $N_V$ nor $N_2$ is better than the other.  Notice that the last example can be generalized as follows.  If $|\alpha|=|\beta|$ with $\alpha_1=\beta_1$ and $\gamma=(\gamma_1)$, then $N_V\leq N_2$.  
\end{exa}

We conclude this section applying our bounds to some interesting examples of
Kronecker coefficients appearing in the literature.

\begin{exa}[The Kronecker coefficients indexed by three hooks]\label{ex:three hooks} Our
  first 
example looks at the elegant situation where the three indexing partitions  are hooks.
 Note that after deleting the first part of a hook we always obtain a one
column shape. Let $\alpha=(1^e)$, $\beta=(1^f)$ and $\gamma=(1^d)$ be the
 reduced partitions, with $d$, $e$ and $f$ positive.
In Theorem 3 of \cite{Rosas:2001},  it was shown
that Murnaghan's inequalities describe the stable value of the Kronecker
 coefficient ${g}_{\alpha[n],\beta[n]}^{\gamma[n]}$,
\[
\overline{g}_{\alpha,\beta}^{\gamma}= (( e \le d+f)) (( d \le e+f)) (( f
\le e+d))
\]
where $((P))$ equals $1$ if the proposition is true, and $0$ if not.

Moreover,
$\stab(\alpha,\beta,\gamma) $ was actually computed  in the proof of Theorem 3 \cite {Rosas:2001}. It was shown that the
 Kronecker coefficient equals $1$ if and only if 
 Murnaghan's inequalities hold, as well as the additional inequality 
 $e+f \le d+2(n-d)-2$. This last inequality says that:
\[
\stab(\alpha,\beta,\gamma) =  \Big[ \frac{d+e+f+3}{2} \Big]=N_2(\alpha,\beta,\gamma)
\]
To summarize, for triples of hooks, Murnaghan's inequalities govern
the value of the reduced Kronecker coefficients, and $N_2$ is a sharp
bound. On the other hand, the bounds provided by $N_1$, $N_B$, and
$N_V$ are not in general sharp.
 \flushright$\Box$

\end{exa}

\begin{exa}[The Kronecker coefficients indexed by two two-row shapes]
\label{ex:two two-row} After deleting the first part of a two-row
  partition we obtain a partition of length $1$.
Let $\alpha$ and $\beta$ be one-row partitions. We have:
\begin{align*}
N_1(\alpha,\beta,\gamma)&=\alpha_1+\beta_1+\gamma_1\\
N_2(\alpha,\beta,\gamma)&=\alpha_1+\beta_1+\gamma_1+\left[ \frac{\gamma_2+\gamma_3}{2}\right]
\end{align*}
It follows from \cite{Briand:Orellana:Rosas:Chamber} that when $\overline{g}_{\alpha,\beta}^{\gamma}>0$,
\[
\stab(\alpha,\beta,\gamma)=\gamma_1-\gamma_3+\alpha_1+\beta_1.
\]
 Neither $N_1$ nor $N_2$ are sharp bounds. Indeed,  for $\overline{g}_{\alpha,\beta}^{\gamma}>0$ we have $\stab(\alpha, \beta,
 \gamma) <N_1$ if $\gamma_3>0$, and $\stab(\alpha, \beta, \gamma) <N_2$ if $\gamma_2 >0$.

Moreover, $N_1<N_2$ when
$\gamma_2+\gamma_3>1$.
 \flushright$\Box$
\end{exa}

\begin{exa}[The Kronecker coefficients: One of the
    partitions is a two-row shape] 
The case when $\gamma$ has only one row, $\gamma=(p)$,  
was studied in \cite{Ballantine:Orellana}. It
  was shown there (Theorem 5.1) that  
  \[
  \stab(\alpha,\beta,(p) ) \le  |\alpha|+\alpha_1+ 2\,p.
  \]
Notice that this bound coincides with $\stab(\alpha,(p))$ after Theorem \ref{thm:global}.
In this case,
\[
N_1=p+|\bar{\alpha} \cap \bar{\beta}|+\alpha_1+\beta_1,
\]
is less than or equal to $N_2$. 
It is also mentioned in \cite{Ballantine:Orellana} that, for the case when $\alpha=\beta$, Vallejo's
bound $N_V$ does beat this bound (that is, $ \stab(\alpha,\alpha) $),  but not always. Indeed, when $\alpha=\beta$, $N_2$ coincides with $N_V$. 
\flushright$\Box$
\end{exa}


The situation described in the previous example, where 
$\stab(\alpha, \beta)<N_V(\alpha, \beta, \gamma)$
 raises the question of
whether  $\min(N_1, N_2)$ is always less or equal to $\stab(\alpha,
\beta)$ when $\overline{g}_{\alpha,\beta}^{\gamma}>0 $. This is indeed the case since, as a direct consequence of Theorem \ref{thm:g and g1} ,$N_2 \leq |\alpha|+|\beta|+\alpha_1+\beta_1$

\begin{exa}[Vallejo's example]\label{exa:Vallejo example} In \cite{Vallejo} the case $\alpha=(3,2)$, $\beta=(2,2,1)$, $\gamma=(2,2)$ was considered. 
In this case $\stab(\alpha,\beta,\gamma)=10$, but 
\[
N_B(\alpha,\beta,\gamma)=N_V(\alpha,\beta,\gamma)=N_1(\alpha,\beta,\gamma)=11.
\]
 Nevertheless, $N_2(\alpha,\beta,\gamma)=10$.  \flushright$\Box$
\end{exa}





\section*{Acknowledgments}

We thank Ernesto Vallejo for pointing to us the reference \cite{Brion:Foulkes},  Ron King for pointing to us Littlewood's formula \eqref{eq:Littlewood}, and Richard Stanley for suggesting to look at \cite{Kirillov:saturation}. We also thank John Stembridge for making freely available his Maple package SF \cite{Stembridge:SF}. 


\bibliographystyle{alpha}


\def\cprime{$'$}

\end{document}